\documentclass[12pt]{article}
\usepackage{pslatex}
\usepackage{fancyhdr}
\usepackage{graphicx}
\usepackage{geometry}
\RequirePackage[latin1]{inputenc} \RequirePackage[T1]{fontenc}

\def\figurename{Figure} 
\makeatletter
\renewcommand{\fnum@figure}[1]{\figurename~\thefigure.}
\makeatother

\def\tablename{Table} 
\makeatletter
\renewcommand{\fnum@table}[1]{\tablename~\thetable.}
\makeatother
\usepackage{color}
                [2000/03/29 v1.0m Standard LaTeX package]

\usepackage{amsmath}
\usepackage{amssymb}
\usepackage{amsfonts}
\usepackage{amsthm,amscd}
\newtheorem{theorem}{Theorem}[section]

\newtheorem{proposition}[theorem]{Proposition}
\theoremstyle{definition}

\newtheorem{definition}[theorem]{Definition}
\newtheorem{example}[theorem]{Example}

\theoremstyle{remark}
\newtheorem{remark}[theorem]{Remark}

\numberwithin{equation}{section}

\def\P{\mathbb P}
\def\R{\mathbb R}
\def\E{\mathbb E}

\def\E{\mathbb E}
\def\N{\mathbb N}

\begin{document}

\title{{Backward stochastic Volterra integral equations with time delayed generators}}

\author{Harouna Coulibaly $^2$ \thanks{harounbolz@yahoo.fr}\;\; and \; Auguste Aman $^1$\thanks{aman.auguste@ufhb.edu.ci/augusteaman5@yahoo.fr, corresponding author}\\
1. UFR Math\'{e}matiques et Informatique, Université Félix H. Boigny, Cocody,\;\;\;\;\;\;\;\;\;\;\\
22 BP 582 Abidjan, C\^{o}te d'Ivoire\;\;\;\;\\
2. Ecole Supérieure Africaine des TIC, Treichville, info@esatic.ci, Côte d'Ivoire\;\;\;\;\;}

\date{}
\maketitle \thispagestyle{empty} \setcounter{page}{1}

\thispagestyle{fancy} \fancyhead{}
 \fancyfoot{}
\renewcommand{\headrulewidth}{0pt}

\begin{abstract}
In this paper, we study backward stochastic Volterra integral equations of type-I with time delayed generators. Under some condition (small time horizon or a Lipschitz constant), we derive an existence and uniqueness results. Next, with the help of two examples of this BSVIEs we provide this condition is necessary. However, for special class of delayed generators, the previous existence and uniqueness result hold for an arbitrary time horizon and Lipschitz constant. We end this paper by establishing under an additive assumptions, a path continuity property.
\end{abstract}

\vspace{.08in} \noindent \textbf{MSC}:Primary: 60F05, 60H15; Secondary: 60J30\\
\vspace{.08in} \noindent \textbf{Keywords}: Backward stochastic Volterra integral equations, time delayed generators, Hölder continuity condition.

\section{Introduction}
\label{intro}
\setcounter{theorem}{0} \setcounter{equation}{0}
This paper is concerned with introducing a new class of backward stochastic Volterra integral equations (type-I BSVIEs, in short) with delayed generator. On a complete filtered probability space $(\Omega ,\mathcal{F} ,\mathbb{P},{\bf F} = (\mathcal{F}_t)_{0\leq t \leq T} )$, supporting an Brownian $W=\{W(t),\;\; t\in [0,T]\}$, and denoting by ${\bf F}$ the $\P$-augmented natural filtration generated by $W$. Given data, that is to say a collection of $\mathcal{F}_T$-measurable random variable $(\psi(t))_{t\geq 0}$, a mapping $f$, referred to respectively as the terminal condition and the generator, we called by backward stochastic Volterra integral equations (type-I BSVIEs, in short) with delayed generator, the equation of the form
\begin{eqnarray}
Y(t) & = & \psi(t) + \int_t^T f(t,s,Y_s,Z_{t,s})ds - \int_t^T Z(t,s) dW(s), \label{Eq1}
\end{eqnarray}
where setting $\Delta=\{(t,s)\in [0,T]^2,\; t\leq s\}$ and for all $(t,s)\in \Delta$ we denote $(Y_s,Z_{t,s} )=(Y(s+u),Z(t+u,s+u))_{-T\leq u \leq 0}$. 

The motivation for this study is two-fold:

From a theoretical point of view, BSVIE \eqref{Eq1} can be viewed as the natural extension of backward stochastic differential equations
\begin{eqnarray}
Y(t) & = & \xi + \int_t^T f(s,Y_s,Z_{s})ds-\int_t^T Z(s) dW(s)\label{Eq2}
\end{eqnarray}
and 
\begin{eqnarray}
Y(t)=\xi+ \int_t^T f(t,s,Y(s),Z(t,s))ds-\int_t^T Z(t,s)dW(s),\label{BSDEV}
\end{eqnarray}
where $\xi$ i a$\mathcal{F}_T$-measurable random variable.

Nowadays \eqref{Eq2} and \eqref{BSDEV} are called in the literature as delayed BSDEs and type-I BSVIEs respectively. The first mention of such equations is, to the best of our knowledge, due to Delong and Imkeller in \cite{P2} and \cite{P3} for delayed BSDEs and Lin and Aman et al. in \cite{P8} and \cite{P9} respectively for type-I BSVIEs.

A few year later, many others work have been done for equations \eqref{Eq2} and \eqref{BSDEV} respectively. Regarding type-I BSVIEs, see \cite{Wal} and \cite{P'4} for result exploiting representation formulae of BSVIEs in the Markovian and non-Markovian  framework, respectively. For type-I BSVIEs driving by discontinuous processes see \cite{Popier}. For so-called extended type-I BSVIEs see \cite{Ha} and \cite{Heral}. Other BSVIEs called type-II BSVIEs were first addressed in \cite{Yong1,Yong2} in the context of optimal control of (forward) stochastic Volterra integral equations (FSVIEs, for short). Next, for type-II BSVIEs see \cite{W} and \cite{Lal} for results in the $L^p$ sense. Concerning delayed BSDE, see \cite{HA} for non-Lipschitz condition and \cite{P1,P'1} in the context of finance and assurance. For reflected delayed BSDE see \cite{Ral} and \cite{Hal}.

The second motivation is practical. Indeed, in the case the value function contains a time memory and delay, it is possible to use BSDE \eqref{Eq1} to derive an optimal control problem or a stochastic differential utility. More precisely, the process $Y$ can be interpret as a price or utility and BSDE \eqref{Eq2} implies the relation
\begin{eqnarray}\label{U}
Y(t) & = & \E\left(\psi(t) + \int_t^T f(s,Y_s,Z_{s})ds|\mathcal{F}_t\right)
\end{eqnarray}
Several work (see e.g \cite{Rozen}, \cite{Bell}, \cite{DR} and \cite{LP}) explained that delay on the value function comes from the notion of disappointment effect on non-monotone preferences modeling due to aversion against volatility. In this context, Delong in \cite{P'1} propose two examples of the value function $f$ defined respectively by
\begin{eqnarray}\label{F1}
f(t,Y_t,Z_t)=-u(c_t)-\beta\left(u(c_t)-\frac{1}{t}\int_0^tY(s)ds\right)
\end{eqnarray}
and 
\begin{eqnarray}\label{F2}
f(t,Y_t,Z_t)=-u(c_t)-\beta\left(u(c_t)-\frac{1}{t}\int_0^tZ(s)ds\right),
\end{eqnarray}
where $u$ and $c$ denote an instantaneous utility and a future consumption stream respectively. To end this paragraph, let us give interpretation of generators \eqref{F1} and \eqref{F2}. Indeed, recalling interpretation of \eqref{U}, we could say that under the generator \eqref{F1}, price (utility) $Y$ for $\xi$ is assumed to change proportionately to the average of the past observed price. Under the infinitesimal conditional expected price change $\displaystyle \beta\frac{1}{t}\int_0^ty(s)ds$, high, on average, past prices imply that the investor expects to trade $\psi$ for a high price in the next day. Under the generator \eqref{F2} the price (utility) $Y$ for $\psi$ is assumed to change proportionately to the average of the past price volatilities. Under the infinitesimal conditional expected price change $\displaystyle \beta\frac{1}{t}\int_0^t z(s)ds$, ds, high, on average, past price volatilities imply that the investor expects to trade $\psi$ for a high price in the next day. These seem to be intuitive feedback relations as far as investors' local beliefs or local pricing rules are concerned.
In view of all above, the relation \eqref{U} can be extend in more general situation as follows
\begin{eqnarray*}
Y(t) & = & \E\left(\psi(t) + \int_t^T f(t,,s,Y_s,Z_{s})ds|\mathcal{F}_t\right)
\end{eqnarray*}
which provides that there exists a real interest to study delayed BSVIEs. Moreover, to our knowledge, this study is new in the literature.
 
The rest of this paper is organized as follows. In Section 2, we introduce some fundamental knowledge and assumptions concerning the data of BSVIE \eqref{Eq1}. Section 3 is devoted to derive existence and uniqueness problems. Finally a regularity (path continuity of the  solution process $Y$) is established in Section 4.

\section{Preliminaries}
\setcounter{theorem}{0} \setcounter{equation}{0}
Let us consider a standard $d$-dimensional Brownian motion $W = (W(t),\; t\geq 0 ) $ defined on a probability space $(\Omega, \mathcal{F},\P)$. We denote by ${\bf F}= (\mathcal{F}_t )_{t\geq0}$ the filtration generated by $W$ and augmented by all $\P$-null sets such the filtered probability space $(\Omega, \mathcal{F}_t,{\bf F },\P)$ satisfies the usual conditions. The Euclidean norm of $ \mathbb{R} $ is denoted by $| \cdot |$. Next, for $ \beta > 0 $, we consider the following spaces.      
\begin{description}
\item $\bullet$ Let $ L^2(\Omega,\mathcal{F}_T,\mathbb{P})$ be the space of $\mathcal{F}_T$-measurable random variables $\xi: \Omega \rightarrow \mathbb{R} $ normed by $\displaystyle \|\xi\|_{L^2}^2=\E(|\xi|^2)$.
\item $\bullet$ Let $ \mathbb{H}_1^{\beta}:= \mathcal{H}_{[-T,T]}^2(\R)$ denote the space of all predictable process $\eta$ with values in $\R$ such that $\displaystyle\E\left(\int_{-T}^T e^{\beta s}|\eta(s)|^2\right)<+\infty$.
\item $\bullet$ Let $ \mathbb{H}_2^{\beta}:= \mathcal{H}_{[0,T]^2}^2(\R)$ denote the space of all process $\varphi(.,.)$ with values in $\R $ such that for all $t\in[0,T],\, \varphi(t,.)$ is predictable process such that $\displaystyle\E\left (\int_0^T\int_0^T e^{\beta s} |\varphi(t,s)|^2dsdt\right)<+\infty$.
\item $\bullet$ Let $ \mathcal{S}^2(\R)$ denote the space of all predictable process $\eta$ with values in $\R$ such that \newline $\E\left(\sup_{0\leq s\leq T}e^{\beta s}|\eta(s)|^2\right)<+\infty$.\end{description}
We endowed the spaces $\mathbb{H}_1^{\beta}$, $\mathbb{H}_2^{\beta}$ and $\mathcal{S}^2(\R)$ respectively with the norm $\|\eta\|^2_{\mathbb{H}_1}=\displaystyle \E\left[\int_{-T}^T e^{\beta s}|\eta(s)|^2ds\right],\newline \displaystyle \|\varphi\|^2_{\mathbb{H}_2}=\E\left[\int_0^T\int_0^T e^{\beta s} |\varphi(t,s)|^2ds dt\right]$ and $\|\eta\|^2_{\mathcal{S}^2}=\E\left[\sup_{0\leq s\leq T}e^{\beta s}|\eta(s)|^2\right]$.
 
We also consider this two additive spaces  
\begin{description}
\item $\bullet $ Let $ L_{-T}^2 (\mathbb{R} ) $ denote the space of measurable functions $ z : [-T;0] \rightarrow \mathbb{R} $ satisfying
$$   \int_{-T}^0 \mid z(t) \mid^2 dt < \infty .
$$
\item $ \bullet $ Let $ L_{-T}^{\infty } (\mathbb{R} )$ denote the space of bounded, measurable functions $ y : [-T,0] \rightarrow \mathbb{R} $\\
satisfying
$$
\sup\limits_{-T\leq t\leq 0} \mid y(t) \mid^2 < \infty.
$$
\end{description}

Before giving our study's framework, let us clarify some notations appearing in this paper. Since BSVIEs considered in this work is a time delay type, we set by $(Y(t),Z(t,s))$, the value of solution process at $(t,s)\in [0,T]^2$ and by $(Y_t,Z_{t,s}) = (Y(t+u),Z(t+u,s+u))_{-T \leq u \leq 0} $ the past of this solution until $(t,s)$. Therefore, for each $(t,s) \in [0,T]^2$ and almost all $\omega \in \Omega,\;\; Y_t(\omega) $ and $ Z_{t,s} (\omega)$ belong respectively to $L_{-T}^{\infty } (\mathbb{R} ) $ and $  L_{-T}^2 (\mathbb{R}) $.

Now, we mentionne that the study of BSVIEs \eqref{Eq1} shall be done under the following assumptions on data.

\begin{description}

\item {(\bf A1)}  $(\psi(t))_{t\geq 0}$ is a stochastic process such that for all $t\geq 0$, $ \psi(t) \in  L^2 ( \Omega,\mathcal{F}_T,\mathbb{P} )    $ .
\item {(\bf A2)} $ f: \Omega \times [0,T]^2\times  L_{-T}^{\infty} (\mathbb{R} ) \times  L_{-T}^2 (\mathbb{R} ) \rightarrow \mathbb{R} $ is a product measurable and $ {\bf F} $-adapted function such that, there exists a probability measure $\alpha$ defined on $([-T,0],\mathcal{B}([-T,0]))$ and a positive constant $K$,
\begin{eqnarray*}
&& \mid f(t,s,y_s,z_{t,s}) - f(t,s,y'_s ,z'_{t,s} )\mid^2 \\
&&\leq  K \left( \int_{-T}^0  \mid y(s+u) - y'(s+u) \mid^2  \alpha (du)\right.\\
&& \left.+ \int_{-T }^0\mid z(t+u,s+u) - z'(t+u,s+u) \mid^2 \alpha (du) \right) ,
\end{eqnarray*}
for $ \P \otimes \lambda $-a.e. $(\omega ,(t,s) ) \in \Omega \times[0,T]^2$ and for any $ (y_t,z_{t,s}),(y'_t,z'_{t,s}) \in  L_{-T}^{\infty} (\mathbb{R} ) \times  L_{-T}^2 (\mathbb{R})$.
\item {(\bf A3)}\\
(i) For $t<0$ or $s<0,\; \; f(t,s,.,.) = 0 $,\\
(ii) $ \displaystyle \mathbb{E} \left[ \int_{[0,T]^2} \vert f(t,s,0,0) \vert^2 dsdt\right]  <  +\infty$.
\end{description}
Next, we need the following slightly stronger assumptions on coefficient to derive a regularity (continuity) of the solution.
\begin{description}
\item {(\bf A4)} $ f: \Omega \times [0,T]^2\times  L_{-T}^{\infty} (\mathbb{R} ) \times  L_{-T}^2 (\mathbb{R} ) \rightarrow \mathbb{R} $ is product measurable and $ {\bf F} $-adapted function satisfying\\
i) For some $ \gamma \in (0,\frac{1}{2} ] $ and $ K'> 0 $, for each $ (y,z) \in  L_{-T}^{\infty} (\mathbb{R} ) \times  L_{-T}^2 (\mathbb{R}) $ we have
\begin{eqnarray*}
\mid f(t',s,y,z) - f(t,s,y ,z )\mid & \leq & K' \mid t'-t\mid ^{\gamma} .
\end{eqnarray*}
ii) 
\begin{eqnarray*}
\displaystyle \E \left[ \left( \int_0^T \mid f(0,s,0,0)\mid^2 ds\right)^{\frac{r}{2}}\right]   < + \infty.
\end{eqnarray*}
\item {(\bf A5)}
\item There exists an $r>\frac{1}{\gamma} $ such that for $t', t\geq 0, \; \psi(t) \in L^r (\Omega,\mathcal{F}_T,\P )$ and 
\begin{eqnarray*}
	\E(|\psi(t)-\psi(t')|^{r})\leq C|t-t'|^r
\end{eqnarray*}
\end{description}
We end this section with this remark which explains some of assumptions above.
\begin{remark}\label{R1}
\begin{itemize}
\item [(a)] Assumption $(\bf A3)$-$(i) $ allows us to take $Y(t)=Y(0)$ and $ Z(t,s)= 0 $ for $t<0$ or $s<0$, as a solution of \eqref{Eq1}.
\item [(b)] The quantity $f(t,s,0,0) $ in $ (\bf A3)$-$(ii) $ should be understood as a value of the generator $f(t,s,y_s,z_{t,s})$ at $y_s=z_{t,s} = 0 $.
\item [(c)] From $ (\bf A4) $, we can prove that
\begin{eqnarray*}
\E \left[  \left( \sup\limits_{0 \leq t \leq T} \int_t^T \mid f(t,s,0,0) \mid^2 ds \right)^{\frac{r}{2}}\right]  & < & +\infty.
\end{eqnarray*}
In particular 
\begin{eqnarray*}
\E \left[  \left( \int_{[0,T]^2} \vert f(t,s,0,0) \vert^2 dsdt \right)^{\frac{r}{2}}\right]&< & +\infty.
\end{eqnarray*}
\item [(d)] Since $ \gamma \in (0,\frac{1}{2} ] $, if $ r > \frac{1}{\gamma} $ then $ r > 2 $. Therefore according to $ (c) $, $ (\bf A4)$-$(ii) $ implies $ (\bf A1) $ and $ (\bf A3)$-$(ii) $.
\end{itemize}
\end{remark}

\section{Problem of existence and uniqueness}
\subsection{Existence and uniqueness result}
In this subsection, we will first establish a result of existence and uniqueness of BSVIEs with time delayed generator when the Lipschitz constant $K$ or a time horizon $T$ are small enough. Before let us state what we mean by solution of BSVIE \eqref{Eq1}.
\begin{definition}
A couple $(Y(.),Z(.,.))$ is called a solution of BSVIE \eqref{Eq1} if it belong to $\mathcal{S}^2\times \mathbb{H}_2$ and satisfies \eqref{Eq1}. 
\end{definition}
\begin{remark}
Since $Z(.,.)\in \mathbb{H}_2$, for all $t\in[0,T], Z(t,.)$ belongs to $L^{2}([0,T],(\mathcal{F}_{s})_{0\leq s\leq T})$. Therefore for all $s\in [t,T],\; \displaystyle \int^T_s Z(t,r)dW(r)$ is a classical Itô stochastic integral with respect the Brownian motion.
\end{remark}
\begin{theorem}\label{T}
Assume $({\bf A1}),\, ({\bf A2})$ and $({\bf A3})$ hold. For a sufficiently small time horizon $T$ or Lipschitz constant $K$, BSVIE admit a unique solution.
\end{theorem}

To prove the Theorem \ref{T}, let us derive and prove this proposition. 

\begin{proposition}\label{P}
Assume $({\bf A1}),\, ({\bf A2})$ and $({\bf A3})$ hold. Let $(Y(.),Z(.,.))$ belong to $ \mathbb{H}_1 \times\mathbb{H}_2$ and satisfy \eqref{Eq1}. Then $Y(.)\in \mathcal{S}^2 $.
\end{proposition}

\begin{proof}
Let $ (Y,Z) $ be a process belong in $ \mathbb{H}_1 \times L^{2}([0,T],\mathbb{H}_2)$ and satisfies \eqref{Eq1}. Applying Itô formula to $ e^{\frac{\beta}{2}t}Y(t),\;\; (\beta > 0) $ and hence taking conditional expectation with respect $\mathcal{F}_t$, we have
\begin{eqnarray*}
&& e^{\frac{\beta }{2}t} Y(t)+\frac{\beta}{2} \E\left(\int_t^T e^{\frac{\beta }{2}s} Y(s) ds|\mathcal{F}_t\right)\\
& = & \E\left(e^{\frac{\beta }{2}T} \psi(t) +\int_t^T e^{\frac{\beta }{2}s}  f(t,s,Y_s,Z_{t,s} ) ds|\mathcal{F}_t\right). 
\end{eqnarray*}
Hence
\begin{eqnarray*}
e^{\frac{\beta}{2}t} \mid Y(t) \mid & \leqslant & e^{\frac{\beta }{2}T}\E \left[ \mid\psi(t) \mid \mid \mathcal{F}_t  \right] + \E \left[ \int_0^T e^{\frac{\beta}{2}s} \mid   f(t,s,Y_s,Z_{t,s} ) \mid ds \mid \mathcal{F}_t  \right].
\end{eqnarray*}
Moreover, Doob's inequality with Cauchy-Schwarz's and Jensen's inequalities yield   
\begin{eqnarray*}
&&\E \left[  \sup\limits_{0 \leqslant t \leqslant T} e^{\beta t} \mid Y(t) \mid^2 \right] \\
& \leq &  8e^{\beta T}\E [ \mid\psi(t) \mid^2 ] + 8T\E \left[ \int_0^T e^{\beta s} \mid  f(t,s,Y_s,Z_{t,s} ) \mid^2 ds \right].
\end{eqnarray*}
In view of $ (\bf A1) $ the first term of the right is finite. To prove that $ Y $ belong to $ \mathcal{S}^2 $, it suffice to provide
\begin{eqnarray*}
\E \left[\int_0^T e^{\beta s} \vert f(t,s,Y_s,Z_{t,s}) \vert^2 ds \right] & < & +\infty ,\;\; \forall  t\in [0,T].
\end{eqnarray*}
This is true if
\begin{eqnarray*}
\displaystyle \E  \left[\int_{[0,T]^2} e^{\beta s} \mid  f(t,s,Y_s,Z_{t,s} ) \mid^2 dsdt\right] & <& + \infty .
\end{eqnarray*}
For that, applying ({\bf A2}) we get
\begin{eqnarray}\label{eq2}
\int_{[0,T]^2} e^{\beta s} \vert f(t,s,Y_s,Z_{t,s}) \vert^2 dsdt 
& \leq &   2K \int_{[0,T]^2}e^{\beta s}\left(\int_{-T}^0 \left(  \left| Y(s+u) \right|^2 + \left| Z(t+u,s+u) \right|^2 \right)   \alpha (du)\right)  dsdt\nonumber\\
&& + 2 \int_{[0,T]^2} e^{\beta s}\vert f(t,s,0,0) \vert^2 dsdt.
\end{eqnarray}
We estimate the first term of right side of \eqref{eq2} as follows.
\begin{eqnarray*}
&& \int_{[0,T]^2} e^{\beta s}\left(\int_{-T}^0 \left(  \left| Y(s+u) \right|^2 + \left| Z(t+u,s+u) \right|^2 \right)   \alpha (du)\right)  dsdt \\
& = & \int_{-T}^0 e^{-\beta u}\left(  \int_{[u,T+u]^2}e^{\beta s} \left( \left| Y(s) \right|^2 + \left| Z(t,s) \right|^2   \right)ds dt     \right) \alpha (du)\\
& \leq & Te^{\beta T} \int_{-T}^T e^{\beta s}\left| Y(s) \right|^2 ds + \int_{[0,T]^2} e^{\beta s}\left|Z(t,s)\right|^2 ds dt, 
\end{eqnarray*}
 We apply Fubini's theorem, change the variables and the fact that $ Z(t,s) =0 $ and $ Y(t) = Y(0) $ for $t<0$ or $s<0$.

Then by putting the above in \eqref{eq2} we obtain
\begin{eqnarray*}
\displaystyle \E  \left[\int_{[0,T]^2} e^{\beta s} \mid  f(t,s,Y_s,Z_{t,s} ) \mid^2 dsdt\right]
& \leq & C \E \left[ \int_{-T}^T e^{\beta s}\left| Y(s) \right|^2 ds + \int_{[0,T]^2} e^{\beta s}\left| Z(t,s) \right|^2 ds dt  \right] \\
&& + 2 \E \left[ \int_{[0,T]^2} e^{\beta s}\vert f(t,s,0,0) \vert^2 dsdt \right] \\
& < & + \infty .
\end{eqnarray*}

\end{proof}

\textbf{Proof of Theorem \ref{T}}\textbf{}
\begin{description}
\item [(i)]\textbf{Existence}\\
We use Picard's iteration to build a sequence of processes and establish that its limit is solution of our BSVIE. Let set $ Y^0 (s) = Z^0 (t,s) = 0 $ and define recursively, for $n\in \mathbb{N} $
\begin{eqnarray}
Y^{n+1} (t) & = & \psi(t) + \int_t^T f(t,s,Y_s^n ,Z_{t,s}^n ) ds - \int_t^T Z^{n+1} (t,s) dW(s) ,\;\; 0 \leqslant t \leqslant T.\label{eq5}
\end{eqnarray}

\textit{Step 1: Given $(Y^n,Z^n)\in \mathbb{H}_1 \times \mathbb{H}_2 $, BSVIEs \eqref{eq5} has a unique solution $(Y^{n+1},Z^{n+1}) \in \mathbb{H}_1\times \mathbb{H}_2 $}.\\
With Proposition \ref{P}, we have for $t\in [0,T]$
\begin{eqnarray*}
\E\int_t^T \vert f(t,s,Y^n_s,Z^n_{t,s}) \vert^2 ds & < & + \infty. 
\end{eqnarray*}
Therefore for a fixed $t\in [0,T]$, the process $ X_t^{n+1} $ define by
\begin{eqnarray*}
X_t^{n+1} (u) = \E \left( \psi(t) + \int_t^T f(t,s,Y_s^n,Z_{t,s}^n) ds \mid \mathcal{F}_u \right) , \;\; u \in [0,T],
\end{eqnarray*}
is a square integrable $ (\mathcal{F}_u) $-martingale and the martingale representation theorem provides a unique process $ Z^{n+1} (t,.) \in \mathbb{H}_2 $ such that
\begin{eqnarray*}
X_t^{n+1} (u) = X_t^{n+1} (0) + \int_0^u Z^{n+1}(t,s) dW(s),\;\; u\in [0,T] .
\end{eqnarray*}
In particular
\begin{eqnarray*}
X_t^{n+1} (T) = X_t^{n+1} (0) + \int_0^T Z^{n+1}(t,s) dW(s)
\end{eqnarray*}
and
\begin{eqnarray*}
X_t^{n+1} (t) = X_t^{n+1} (0) + \int_0^t Z^{n+1}(t,s) dW(s) .
\end{eqnarray*}
For $t\in [0,T]$, if we set $ Y^{n+1} (t) = X_t^{n+1} (t) $, then
\begin{eqnarray}\label{Eq6}
Y^{n+1} (t) = X_t^{n+1} (T) - \int_t^T Z^{n+1}(t,s) dW(s).
\end{eqnarray}
We have also
\begin{eqnarray*}
X_t^{n+1} (T) = \psi(t) + \int_t^T f(t,s,Y_s^n,Z_{t,s}^n) ds,
\end{eqnarray*}
which with \eqref{Eq6} prove that $(Y^{n+1},Z^{n+1} )$ satisfies \eqref{eq5}. Moreover,
\begin{eqnarray*}
\E \left(\int_{-T}^T \mid Y^{n+1} (t) \mid^2 dt   \right) 
& = & \E \left[ \int_{-T}^T \left| \E \left( \psi(t) + \int_t^T f(t,s,Y_s^n,Z_{t,s}^n) ds \mid \mathcal{F}_t\right)  \right|^2 dt   \right]\\
& \leq &  \int_{-T}^T \E \left( \left| \psi(t) + \int_t^T f(t,s,Y_s^n,Z_{t,s}^n) ds  \right|^2  \right)dt \\
& \leq & 2\E \left(\int_0^Y \mid \psi(t) \mid^2dt  \right) + 2 \E\int_{[0,T]^2} \left| f(t,s,Y_s^n,Z_{t,s}^n)  \right|^2 ds dt.\\
& <  & + \infty.   
\end{eqnarray*}
Then, in view of Proposition \ref{P}, $ Y^{n+1} $ belongs to $ \mathcal{S}^2 $, which finally provides that $(Y^{n+1},Z^{n+1} )$ solves \eqref{eq5}.\\

\textit{Step 2: The sequence $(Y^n,Z^n)$ converges in $ \mathbb{H}_1 \times \mathbb{H}_2 $}.\\
For $ (t,s) \in [0,T]^2$, and in virtue of \eqref{eq5} $(\bar{Y}^n(t),\bar{Z}^n(t,s))= (Y^{n+1}(t) -Y^n (t),Z^{n+1}(t,s) -Z^n (t,s))$ satisfies equation
\begin{eqnarray*}
\bar{Y}^n(t) &=& \int_t^T \bar{f}_n (t,s)ds-\int_t^T \bar{Z}^n(t,s) dW(s),\;\; 0\leq t \leq T,
\end{eqnarray*} 
 where $\bar{f}_n (t,s)=f(t,s,Y_s^n,Z_{t,s}^n) -  f(t,s,Y_s^{n-1},Z_{t,s}^{n-1})$.
 
 For any $\beta > 0$, it follows from Itô's formula applying to $ e^{\beta t}\mid\bar{Y}^n(t) \mid^2$
\begin{eqnarray}\label{A1}
\E\left(e^{\beta t}\mid \bar{Y}^n(t)\mid^2 +\int_t^Te^{\beta s}|\bar{Z}^n(t,s)|^2ds\right) 
&=&-\beta\E\left(\int_t^Te^{\beta}|\bar{Y}^n(s)|^2ds+2\int_t^T\bar{Y}^n(s)\bar{f}^n(s)ds\right)\nonumber\\
& \leq & \frac{1}{\beta} \E\int_t^T e^{\beta s}\mid\bar{f}_n (t,s)  \mid^2 ds,
\end{eqnarray}
where we use inequality $2ab\leq \beta a^2+\frac{1}{\beta}b^2$.
Hence
\begin{eqnarray}\label{Eq7}
\E \left[ \int_{-T}^T e^{\beta t}\mid \bar{Y}^n(t) \mid^2 dt\right] & \leq &  \frac{1}{\beta} \E \left(  \int_{[0,T]^2} e^{\beta s}\mid  \bar{f}_n (t,s)\mid^2 dsdt \right).
\end{eqnarray}
On the other hand, with \eqref{A1}, Theorem 2.2 in \cite{Yong1} and \eqref{Eq7}, we have
\begin{eqnarray}\label{A2}
\E\left(\int_{[0,T]^2}e^{\beta s}|\bar{Z}^n(t,s)|^2ds\right) 
& \leq & \frac{1}{\beta} \E\left(\int_{[0,T]^2}e^{\beta s}\mid\bar{f}_n (t,s)  \mid^2 dsdt\right)+\E\left(\int_0^T\int_0^t e^{s\beta}|\bar{Z}^n(t,s)|^2dsdt\right)\nonumber\\
&\leq & \frac{1}{\beta} \E\left(\int_{[0,T]^2}e^{\beta s}\mid\bar{f}_n (t,s)\mid^2 dsdt\right)+2\E\left(\int_{[0,T]}e^{\beta t}\mid \bar{Y}^n(t)\mid^2dt\right)\nonumber\\
&\leq & \frac{3}{\beta}\E\left( \int_{[0,T]^2}e^{\beta s}\mid\bar{f}_n (t,s)\mid^2 dsdt\right).
\end{eqnarray}
Finally with $(\bf A2)$, Fubini's theorem, change of variables
and knowing that $Y^n(t) = Y^n(0) $ and $Z^n(t,s) =0$ for $t<0$ or $s<0$, we obtain
\begin{eqnarray*}
&&\E\left(\int_{-T}^{T}e^{\beta t}\mid \bar{Y}^n(t)\mid^2 dt+\int_{[0,T]^2}e^{\beta s}|\bar{Z}^n(t,s)|^2dsdt\right)\\
 &\leq & \frac{4}{\beta} \E\left[\int_{[0,T]^2} e^{\beta s}  \mid  \bar{f}_n (t,s)  \mid^2 dsdt  \right] \\ 
& \leq & \frac{4K}{\beta}  \E \left[ \int_0^T\int_0^T e^{\beta s}  \int_{-T}^0  \mid \bar{Y}^{n-1} (s+u)\mid^2 \alpha (du) dsdt \right] \\
&& +\frac{4K}{\beta} \E\left[\int_0^T \int_0^T e^{\beta s}  \int_{-T}^0  \mid \bar{Z}^{n-1} (t+u,s+u))\mid^2\alpha (du)dsdt \right] \\& \leq &  \frac{4K}{\beta}  \int_{-T}^0 e^{-\beta u} \alpha (du) \E \left[ T \int_{-T}^T e^{\beta s}\mid \bar{Y}^{n-1} (s) \mid^2 ds + \int_{[0,T]^2} e^{\beta s} \mid \bar{Z}^{n-1}(t,s) \mid^2  ds dt\right].   
\end{eqnarray*}
Taking $\beta = \frac{1}{T}$ we have by iterative argument
\begin{eqnarray}
&& \E \left[ \int_{-T}^T e^{\beta t}\mid \bar{Y}^n(t) \mid^2 dt + \int_{[0,T]^2} e^{\beta s}\mid \bar{Z}^n(t,s) \mid^2 ds dt \right] \nonumber\\
&& \leq \left(4TKe \max (1,T)\right)^n \E \left[ \int_{-T}^T e^{\beta t}\mid Y^{1} (t) \mid^2 ds + \int_{[0,T]^2} e^{\beta s} \mid Z^{1}(t,s) \mid^2  ds dt\right].\label{part}  
\end{eqnarray}
Suppose that $4TKe \max (1,T)<1$ i.e $T$ or $K$ be taken sufficiently small, then $(Y^n,Z^n)_{n\geq 1} $ is a Cauchy sequence in the Banach space $\mathbb{H}_1 \times \mathbb{H}_2 $. Consequently there exists a unique process $(Y,Z) \in \mathbb{H}_1 \times \mathbb{H}_2 $ such that
\begin{eqnarray*}
\E \left[ \int_{-T}^T e^{\beta t}\mid Y^n(t)-Y(t) \mid^2 dt + \int_{[0,T]^2} e^{\beta s}\mid Z^n(t,s) - Z(t,s) \mid^2 ds dt \right]\rightarrow 0,\;\; \rm{as} \;\; n \rightarrow + \infty.
\end{eqnarray*}

\textit{Step 3: The process $(Y,Z)$ solves BSVIEs \eqref{Eq1}}.\\
Since $(Y,Z)$ belongs to $\mathbb{H}_1 \times \mathbb{H}_2$ it follows from Proposition \ref{P} that $Y\in \mathcal{S}^2$. By taking the limits for \eqref{eq5}, one easily finds that $(Y,Z)$ solves \eqref{Eq1}.\\
\item [(ii)]\textbf{Uniqueness}\\
Let $(Y,Z)$ and $(Y',Z')$ be two solutions of BSVIEs \eqref{Eq1}. The process $(\Delta Y(t),\Delta Z(t,s))=(Y(t) -Y'(t),Z(t,s) -Z'(t,s))$ satisfies BSVIE
\begin{eqnarray*}
\Delta Y(t) = \int_t^T  \Delta f(t,s)ds - \int_t^T \Delta Z(t,s)dW(s),
\end{eqnarray*}
where $\Delta f(t,s) = f(t,s,Y_s,Z_{t,s})- f(t,s,Y'_s,Z'_{t,s})$.
Using the existence proof argument we get
\begin{eqnarray*}
&& \E \left[ \int_{-T}^T e^{\beta t}\mid \Delta Y(t) \mid^2 dt + \int_{[0,T]^2} e^{\beta s}\mid \Delta Z(t,s)\mid^2 ds dt \right] \\
&& \leq 4TKe \max (1,T) \E \left[  \int_{-T}^T e^{\beta t}\mid \Delta Y(t) \mid^2 dt + \int_{[0,T]^2} e^{\beta s}\mid \Delta Z(t,s) \mid^2 ds dt    \right].  
\end{eqnarray*}
Since $4TKe \max (1,T)< 1$ who is satisfied if the horizon time $T$ or a Lipschitz constant $K$ are small enough, we get 
\begin{eqnarray*}
\E \left[ \int_{-T}^T e^{\beta t}\mid \Delta Y(t) \mid^2 dt + \int_{[0,T]^2} e^{\beta s}\mid \Delta Z(t,s)\mid^2 ds dt \right] = 0.
\end{eqnarray*}
Finally, this implies that $\Delta Y(t)=0$ and $\Delta Z(t,s)=0$. The
proof is now complete.
\end{description}
Given the contravening conditions for the validation of Theorem 3.1, it is natural to ask if it is possible to obtain with other calculation techniques an existence and uniqueness result for fairly large terminal times and large Lipschitz constants. As an answer of this question, we can say that such extension is not possible. However, for a special class of generators independent of $y$ and where the measure of delay is supported by sufficiently small interval time, Theorem 3.1 may be generalized, as we provide now.
\begin{theorem}\label{T1}
Assume $({\bf A1}),\, ({\bf A2})$ and $({\bf A3})$ hold such that the generator is independent of $y_t$, i.e. for $t\in [0,T]$ we have $f(t,s,y_s,z_{t,s}) = g(t,s,z_{t,s})$. Let the measure $\alpha$ be supported by the interval $[-\gamma, 0]$. For a sufficiently small time delay $\gamma$, the backward stochastic differential equation \eqref{Eq1} has a unique solution $(Y,Z)\in\mathcal{S}^2\times \mathbb{H}_2$.
\end{theorem}
\begin{proof}
We unfold the same argument of Theorem \ref{T}. But due to the special form of generator, inequality \eqref{part} becomes
\begin{eqnarray*}
\E\left[\int_{-\gamma}^Te^{\beta s}\mid \bar{Y}^n(t) \mid^2dt+\int_{[0,T]^2} e^{\beta s}\mid \bar{Z}^n(t,s) \mid^2 ds dt \right]
& \leq & \dfrac{K}{\beta} \int_{-\gamma}^0 e^{-\beta u} \alpha (du)\E\left[ \int_{[0,T]^2} e^{\beta s} \mid \bar{Z}^{n-1}(t,s) \mid^2  ds dt\right]\\
&\leq & \dfrac{K}{\beta}e^{\beta\gamma}\E\left[ \int_{[0,T]^2} e^{\beta s}\mid\bar{Z}^{n-1}(t,s)\mid^2 ds dt\right].
\end{eqnarray*}
Since $\gamma$ is supposed small enough, one can take $\beta$ sufficiently big such that $\displaystyle\dfrac{K}{\beta}e^{\beta\gamma}$ will be smaller than $1$. This proves the convergence of $(Y^n,Z^n)_{n\in\N}$.
\end{proof}

\subsection{Non-existence and multiple solutions}

This subsection is devoted to confirm that BSVIEs with delayed generator can not have a unique solution in general condition of terminal time and Lipschitz condition. Therefore, a general extension of Theorem 3.1 is not possible. Indeed, we deal with two examples of BSVIEs whose have no solution or multiple solutions. 
\begin{example}
Let $\alpha$ be a Dirac measure at $-T$ and $K$ a positive constant. We consider the following BSVIE
\begin{eqnarray}\label{F1}
Y(t)= \psi(t) + \int_t^T \int_{-T}^{0}\frac{K}{t+1}Y(s+u){\bf 1}_{[0,+\infty)}(s)\alpha (du)ds -\int_t^T Z(t,s) dW(s),\;\; 0 \leq t\leq T.
\end{eqnarray}
More precisely according to the definition of Dirac measure, BSVIE \eqref{F1} can be rewritten as
\begin{eqnarray*}
Y(t)= \psi(t)+\int_t^T \frac{K}{t+1}Y(s-T)ds -\int_t^T Z(t,s)dW(s),\;\; 0 \leq t\leq T.
\end{eqnarray*}
Moreover, since $Y(t)=Y(0), \; \forall\, t<0$, we derive easily that BSVIE \eqref{F1} becomes
\begin{eqnarray}\label{F2}
Y(t)= \E(\psi(t)|\mathcal{F}_t) +  K\frac{T-t}{t+1}Y(0),\;\; 0 \leq t\leq T,
\end{eqnarray} 
such that for $t=0$ we have
\begin{eqnarray}
(1-TK)Y(0)=\E(\psi(0)).\label{i0}
\end{eqnarray}
Three cases open to us.
\begin{description}
\item [(a)] $ TK < 1 $.\\
For all $t\in [0,T]$, since $\E(\psi(t)|\mathcal{F}_t) \in L^2 ( \Omega,\mathcal{F}_t,\mathbb{P})$, it follows from the martingale representation that there exist a unique square integrable process $Z(t,.)$ such that 
\begin{eqnarray*}
\E(\psi(t)|\mathcal{F}_t) =  \E \left[ \psi(t)\right] + \int_0^{t} Z(t,s) dW(s).
\end{eqnarray*}
 Then according to \eqref{F2} and \eqref{i0}, we have
\begin{eqnarray}\label{F3}
Y(t) & = & \E \left[\psi(t)\right] +\frac{T-t}{t+1}KY(0) +  \int_0^t Z(t,s) dW(s) \nonumber\\
& = & \E \left[\psi(t)\right]+\frac{(T-t)K}{(t+1)(1-KT)}\E(\psi(0)) + \int_0^t Z(t,s) dW(s),\;\; 0 \leq t\leq T .
\end{eqnarray}
Finally, $(Y,Z)$ is the unique process that solves \eqref{F1}. Indeed, if we assume that there is another solution $ (\tilde{Y},\tilde{Z}) $, we get, in view of \eqref{i0}, $(1-T K)Y(0) = \E \left[\psi(0)\right] = (1-TK)\tilde{Y}(0) $ and then $Y(0)=\tilde{Y}(0)$. On the other hand we obtain
\begin{eqnarray*}
\int_0^t(Z(t,s)-\tilde{Z}(t,s)) dW(s) = 0, \;\; \mathbb{P}-a.s.,
\end{eqnarray*}
hence $Z=\tilde{Z}$ and $Y = \tilde{Y}$.
\item [(b)] $TK = 1$ and $\E\left[\psi(0)\right]\neq 0 $.\\
The condition $\E\left[\psi(0) \right] = (1-TK)Y(0) $ is not satisfied and therefore equation \eqref{F1} does not have any solution.
\item [(c)] $TK = 1$ and $\E\left[\psi(0)\right]= 0 $.\\
For $t\in [0,T]$, let recall $Z(t,.)$ appearing in the martingale representation of $\E(\psi(t)|\mathcal{F}_t) $ and consider, as in (a), the process $Y$ defined by
\begin{eqnarray}\label{F4}
Y(t) =\E \left[\psi(t)\right] +\frac{T-t}{t+1}KY(0) +  \int_0^t Z(t,s) dW(s),\;\; 0\leq t\leq T,
\end{eqnarray}
where $Y(0)$ is a arbitrary $\mathcal{F}_0 $-measurable random variable. Therefore any process $(Y,Z)\in \mathbb{S}^2(\mathbb{R}) \times \mathbb{H}^2(\mathbb{R}) $ satisfying \eqref{F4} solves \eqref{F1}.
\end{description}	
\end{example}

\begin{example}
Let $K\in \mathbb{R}$, $\varphi(.)$ a given function from $[0,T]$ to $L^2(\Omega, \mathcal{F}_T)$ and $\alpha$ designed a uniform measure on $[-T,0]$. We consider BSVIEs
\begin{eqnarray*}
Y(t)= \varphi(t) + \int_t^T \int_{-T}^0 K Y(s+u){\bf 1}_{\{ s+u\geq 0\}}(u)\alpha (du){\bf 1}_{\{t\leq s\}}(s)ds -\int_t^T Z(t,s) dW(s),\;\;  0\leq t\leq T. 
\end{eqnarray*}
Using the definition of uniform measure on $[-T,0]$, previous BSVIEs becomes 
\begin{eqnarray}\label{F5}
Y(t) =\varphi(t) +\frac{K}{T}\int_t^T \int_0^{s}Y(u)du ds -\int_t^T Z(t,s) dW(s), \;\; 0 \leq t\leq T.
\end{eqnarray}
Next
\begin{eqnarray} \label{F5bis}
Y(0) =\varphi(0)+\frac{K}{T}\int_0^T \int_0^{s}Y(u) du ds -\int_0^T Z(0,s) dW(s).
\end{eqnarray}

In order to simplify the calculations, we suppose $\varphi$ satisfy $\varphi(t)= e^{-t}\varphi(0)$, for all $t\in [0,T]$ which holds for example if $\varphi(t)= e^{T-t}\xi$. In view of \eqref{F5bis}, we can rewrite \eqref{F5} as follow: for $0\leq t\leq T$,
\begin{eqnarray*}
Y(t)&=& e^{-t}\varphi(0)+\frac{K}{T}\int_t^0\int_0^{s}Y(u)duds+\frac{K}{T}\int_0^T\int_0^{s}Y(u)duds-\int_t^T(Z(t,s)dW(s)-\int_0^T Z(0,s)dW(s)\\
&=&Y(0)+(e^{-t}-1)\varphi(0)-\frac{K}{T}\int_0^t\left(\int ^s_0Y(u)du\right)ds-\int_t^T(Z(t,s)dW(s)+\int_0^TZ(0,s))dW(s)\\
&=& Y(0)+(e^{-t}-1)\varphi(0)-\frac{K}{T}\int_0^t(t-s)Y(s)ds-\int_t^T(Z(t,s)dW(s)+\int_0^TZ(0,s))dW(s).
\end{eqnarray*}
Let consider its deterministic version 
\begin{eqnarray}\label{F7}
y(t) = y(0)+(e^{-t}-1)\varphi(0)-\frac{K}{T}\int_0^t(t-s)y(s)ds+h(t),
\end{eqnarray}
with $h$ belongs to $\mathcal{C}^2 (\mathbb{R})$ such that $h(0)=0$ and $y(0)$ a given initial condition. We obtain the non homogeneous linear second order differential equation
\begin{eqnarray}\label{F8}
\ddot{y} (t)+\frac{K}{T}y(t) = e^{-t}\varphi(0)+ \ddot{h}(t), 
\end{eqnarray}
whose homogeneous solution is defined by
\begin{eqnarray*}
y(t) = A\, e^{\beta t} + B\, e^{-\beta t},
\end{eqnarray*}
where for $K>0,\;\beta=\sqrt{-\frac{K}{T}}$ is seen as a complex number and $A,\; B$ are constants. Next, with similar steps used in \cite{P2}, we derive
\begin{eqnarray*}
Y(t)&=&\frac{Y(0)}{2}(e^{\beta t} + e^{-\beta t})+\frac{\varphi(0)}{2\beta}\left(\frac{e^{\beta t}}{\beta +1}+\frac{e^{-\beta t}}{\beta-1}\right)-\frac{e^{-t}}{(\beta^2-1)}\varphi(0)\\
&&-\int_t^TZ(t,s)dW(s)+\int_0^TZ(0,s)dW(s)\nonumber\\
&&+\frac{\beta}{2}\left[e^{\beta t}\int_0^t \int_t^TZ(s,u))dW(u) e^{-\beta s}ds+e^{-\beta t}\int_0^t \int_t^TZ(s,u)dW(u) e^{\beta s}ds \right]\nonumber\\
&&+\frac{\beta}{2}\left[e^{\beta t}\int_0^t \int_0^TZ(0,u))dW(u) e^{-\beta s}ds-e^{-\beta t}\int_0^t \int_0^TZ(0,u)dW(u) e^{\beta s}ds\right]\nonumber\\
&=&\frac{Y(0)}{2}(e^{\beta t} + e^{-\beta t})+\frac{\varphi(0)}{2\beta}\left(\frac{e^{\beta t}}{\beta +1}+\frac{e^{-\beta t}}{\beta-1}\right)-\frac{e^{-t}}{(\beta^2-1)}\varphi(0)\\
&&+\frac{1}{2}\int_0^T(e^{\beta(t-s)}+e^{-\beta(t-s)})Z(0,s)dW(s)\nonumber\\
&&+\int_t^T\left(-Z(t,s)+\frac{\beta}{2}\int_0^{t}
(e^{\beta(t -u)}+e^{-\beta(t-u)})Z(u,s)du\right)dW(s).
\end{eqnarray*}
Taking conditional expectation which respect $\mathcal{F}_t$, and since the processus $Y$ is $(\mathcal{F}_t)_{t\geq 0}$-adapted, we have
\begin{eqnarray}\label{Vol}
Y(t)&=&\frac{Y(0)}{2}(e^{\beta t} + e^{-\beta t})+\frac{\varphi(0)}{2\beta}\left(\frac{e^{\beta t}}{\beta +1}+\frac{e^{-\beta t}}{\beta-1}\right)-\frac{e^{-t}}{(\beta^2-1)}\varphi(0)\nonumber\\
&&+\frac{1}{2}\int_0^t(e^{\beta(t-s)}+e^{-\beta(t-s)})Z(0,s)dW(s)	
\end{eqnarray}
which implies  
\begin{eqnarray*}
\E \left[\varphi(T)\right]=\frac{Y(0)}{2}(e^{\beta T} + e^{-\beta T})+\frac{\E(\varphi(0))}{2\beta}\left(\frac{e^{\beta T}}{\beta +1}+\frac{e^{-\beta T}}{\beta-1}\right)-\frac{e^{-T}}{(\beta^2-1)}\E(\varphi(0)). 
\end{eqnarray*}
From the property of $\varphi$, we have
\begin{eqnarray*}
\frac{Y(0)}{2}=\left[\frac{\beta^2}{\beta^2-1}e^{-T}-\frac{1}{2\beta}\left(\frac{e^{\beta T}}{\beta+1}+\frac{e^{-\beta T}}{\beta-1}\right)\right]\frac{\E(\varphi(0))}{e^{\beta T}+e^{-\beta T}}
\end{eqnarray*}
\begin{description}
\item [(1)] Assume $K<0$. Then the unique solution $(Y,Z)\in \mathcal{S}^2\times\mathbb{H}_2$ of \eqref{F5} is given by
\begin{eqnarray*}
Y(t) &=& \left[\frac{\beta^2}{\beta^2-1}e^{-T}-\frac{1}{2\beta}\left(\frac{e^{\beta T}}{\beta+1}+\frac{e^{-\beta T}}{\beta-1}\right)\right]\frac{e^{\beta t}+e^{-\beta t}}{e^{\beta T}+e^{-\beta T}}\E(\varphi(0))\\
&&+\frac{\varphi(0)}{2\beta}\left(\frac{e^{\beta t}}{\beta +1}+\frac{e^{-\beta t}}{\beta-1}\right)-\frac{e^{-t}}{(\beta^2-1)}\varphi(0)+\int_0^tZ(t,s)dW(s),
\end{eqnarray*}
such that according to Theorem 2.2 in \cite{Yong1}, the process $Z$ is defined by
\begin{eqnarray*}
Z(t,s)&=& \frac{1}{2}(e^{\beta(t-s)}+e^{-\beta(t-s)})Z(0,s)
\end{eqnarray*}
 where $Z(0,s))$ is defined such $(Y(0),Z(0,.)$ satisfies BSVIE \eqref{F5}.
\item [(2)] Assume now $K>0$. This case is very interesting since, it gives both uniqueness, non existence and multiplicity of solutions.
From Euler's formula, \eqref{Vol} is equivalent to  
\begin{eqnarray*}
Y(t)&=&Y(0)\cos(\beta t)+\frac{\varphi(0)}{\beta^2-1}\left(\cos(\beta t)-i\frac{\sin(\beta t)}{\beta}\right)-\frac{e^{-t}}{(\beta^2-1)}\varphi(0)\\
&&+\int_0^t\cos(\beta(t-s))Z(0,s)dW(s),\label{V1}
\end{eqnarray*}
which implies
\begin{eqnarray*}
e^{-T}\E \left[\varphi(0)\right]=Y(0)\cos(\beta T)+\frac{\E(\varphi(0))}{\beta^2-1}\left(\cos(\beta T)-i\frac{\sin(\beta T)}{\beta}-e^{-T}\right).	
\end{eqnarray*} 
\begin{itemize}
\item [(a)] $\beta T<\frac{\pi}{2}$.
A unique solution $(Y,Z)\in \mathbb{S}^2 (\mathbb{R} )\times \mathbb{H}^2 (\R)$ of \eqref{F5} is given by
\begin{eqnarray*}
Y(t)&=&\frac{1}{\beta^2-1}\left(\beta^2e^{-T}-\cos(\beta T)+i\frac{\sin(\beta T)}{\beta}\right)\frac{\cos(\beta t)}{\cos(\beta T)}\E(\varphi(0))\\
&&+\frac{\varphi(0)}{\beta^2-1}\left(\cos(\beta t)-i\frac{\sin(\beta t)}{\beta}\right)-\frac{e^{-t}}{(\beta^2-1)}\varphi(0)\\
&&+\int_0^t\cos(\beta(t-s))Z(0,s)dW(s).
\end{eqnarray*}
\item [(b)] $\beta T=\frac{\pi}{2}$ and $\E\left[\varphi(0)\right]\neq 0$.\\
Since $(\beta^2e^{-T}+i\frac{\sin(\beta T)}{\beta})\E\left[\varphi(0)\right]=0$ is not satisfied, Equation \eqref{F5} does not have any solution.
\item [(c)] $\beta T=\frac{\pi}{2}$ and $\E\left[\varphi(0)\right]=0$.\\
In this case one can see that BSVIE \eqref{F5} may not have any solution, or may have multiple solutions. Indeed, since $\E\left[\varphi(0)\right]=0$, for an arbitrary process $Z(0,.)$  
\begin{itemize}
\item [(i)] If  
\begin{eqnarray*}
 \E \left[\int_0^T\mid\cos(\beta(T-s))Z(0,s)\mid^2 ds\right]=+\infty,
 \end{eqnarray*}
then equation \eqref{F5} does not have any solution.
\item [(ii)] On the other hand, if 
\begin{eqnarray*}
\E \left[ \int_0^T \mid \cos(\beta(T-s))Z(0,s)\mid^2 ds\right] <+\infty,
\end{eqnarray*}
then equation \eqref{F5} has multiple solutions $(Y,Z)\in \mathbb{S}^2(\mathbb{R})\times\mathbb{H}^2(\mathbb{R})$ given by
\begin{eqnarray}
Y(t)&=& Y(0)\cos(\beta t)-\frac{e^{-(1+\beta )t}}{\beta(\beta^2-1)}\varphi(0)+\int_0^t\cos(t-s)Z(0,s)dW(s).\label{V1}
\end{eqnarray}
with an arbitrary $Y(0)\in L^2(\mathbb{R})$. To end our example let us take $K=T$ and consider  
\begin{eqnarray*}
Z(0,s)= \frac{1}{\cos (\frac{\pi}{2}-s)},
\end{eqnarray*}
then since $\beta T=\frac{\pi}{2}$, we derive easily that $T=\frac{\pi}{2}$ and 
\begin{eqnarray*}
 \E \left[\int_0^T\mid\cos(\beta(T-s))Z(0,s)\mid^2 ds\right]&=&\frac{\pi}{2}
 \end{eqnarray*}
so that we have multiple solutions. Whereas for 
\begin{eqnarray*}
Z(0,s)= \frac{1}{\cos^2 (\frac{\pi}{2}-s)},
\end{eqnarray*}
we don't have any solution, since
\begin{eqnarray*}
\E \left[\int_0^{\frac{\pi}{2}} \left|\frac{1}{\cos (\frac{\pi}{2}-s)} \right|^2 ds \right] = +\infty.
\end{eqnarray*}
\end{itemize}
\end{itemize}
\end{description}
\end{example}

\section{Path regularity for solution of BSVIEs}
We are now ready to give the second main result of this paper which concern the regularity of the solution of BSVIEs \eqref{Eq1}.
\begin{theorem}
Assume $({\bf A2})$, $({\bf A4})$ and $({\bf A5})$ hold. For a sufficiently small time horizon $T$ or Lipschitz constant $K$, BSVIE \eqref{Eq1} admits a unique solution $(Y,Z)$. Moreover, the map $[0,T]\ni t\rightarrow Y(t) $ is continuous.
\end{theorem}

\begin{proof}
According to Remark \ref{R1}, $({\bf A4})$ implies $ ({\bf A1})$ and $ ({\bf A3})$-$ (ii) $. Then existence and uniqueness are valid in virtue of Theorem \ref{T}. It remains to only treat the continuity of the map $ t \rightarrow Y(t) $. This proof will be divided into two steps. In the sequel $C $ will denote a constant which may change from line to line.

\textit{Step 1: Special BSVIE \eqref{Eq1} where $f$ is independent  $y$ and $z$.}\\
In this context, we have
\begin{eqnarray}
Y(t) & = & \psi(t) + \int_t^T f(t,s)ds - \int_t^T Z(t,s) dW(s). \label{Eq01}
\end{eqnarray}  
Define
\begin{eqnarray*}
X_t(u)=\E\left(\psi(t)+\int^T_tf(t,s)ds|\mathcal{F}_u\right),
\end{eqnarray*}
it not difficult to prove that for a fixed $t\in [0,T]$, le processus $(X_t(u))_{u\geq 0}$ is a square integral martingale. Hence, there exist a unique processus $Z(t,.)$ such that
\begin{eqnarray*}
X_t(u)=\E\left(X_t(0)\right)+\int^T_tZ(t,s)dW(s).
\end{eqnarray*}
Moreover, by Doob's maximal inequality, Hölder inequality and $({\bf A5})$, we have 
\begin{eqnarray}\label{ZZZ}
\E\left(\sup_{0\leq t\leq T}X_t(u)-X_t'(u)|^r\right)&\leq &\E\left(|X_t(T)-X_{t'}(T)|^r\right)\nonumber\\
&\leq &C\E\left(|\psi(t)-\psi(t)|^r+\left|\int^T_tf(t,s)ds-\int^T_{t'}f(t',s)ds\right|^r\right)\nonumber\\
&\leq & C\E\left(|\psi(t)-\psi(t)|^r)+\left|\int_{t}^{t'}f(t,s)ds\right|^r+\left|\int_{t'}^{T}(f(t,s)-f(t',s))ds\right|^r\right)\nonumber\\
&\leq & C\left(|t-t'|^r+\E\left(\int_{t}^{t'}|f(t,s)|^2ds\right)^{r/2}|t-t'|^{r/2}+|t-t'|^{\gamma r}\right)\nonumber\\
&\leq &C|t-t'|^{\gamma r}.
\end{eqnarray}
Thanks to $\gamma r>1$, by Kolmogorov's criterion, we know that $(u,t)\mapsto X_t(u)$ is bicontinuous. In particular, $t\mapsto X_t(t)=Y(t)$ is continuous.

On the other hand, by BDG's inequality, it follows from \eqref{ZZZ} that
\begin{eqnarray*}
\E \left[ \left( \int_0^T \mid Z(t,s)-Z(t',s) \mid^2 ds\right)^{\frac{r}{2}}\right]  & \leq & C\E \left[ \left| \int_0^T (Z(t,s)-Z(t',s))dW(s)\right|^{r}\right]\nonumber\\
&\leq & C\E\left(|X_t(0)-X_t(0)|^r+|X_{t}(T)-X_{t'}(T)|^r\right)\nonumber\\
&\leq & C\mid t-t' \mid^{\gamma r}.
\end{eqnarray*}
Moreover, using again assumption $({\bf A5})$ together with , we also get
\begin{eqnarray*}
\E \left[\left(\int_0^T \mid Z(0,s) \mid^2 ds\right)^{\frac{r}{2}} \right] &\leq & C\E \left[\left|\int_0^T Z(0,s)dW(s)\right|^{r}\right]\nonumber\\
&\leq & C\E(|X_0(T)-X_0(0)|^r)\nonumber\\
&\leq & C\E\left[|\psi(0)|^r+\left(\int^T_0|f(0,s)|^2ds\right)^{r/2}\right]\nonumber\\
&\leq &C.
\end{eqnarray*}
Thus if we set $g(t)=Z(t,.)$, then the $t\mapsto f(t)$ can be considered as a random process in Hilbert space $L^2([0,T],ds)$. Therefore estimate \eqref{ZZZ} means that
$$ \E\left(\|f(t)-f(t')\|^r_{L^2}\right)\leq C|t-t'|^{\gamma r}.$$ Hence it follows again from Kolmogorov criterion that
\begin{eqnarray*}
\E \left[\left(\sup_{0\leq t\leq T}\int_0^T \mid Z(t,s) \mid^2 ds\right)^{\frac{r}{2}} \right] &\leq & C\E\left(\|f(t)\|^r_{L^2}\right)\nonumber\\
&\leq & C.
\end{eqnarray*}
 Finally, 
\begin{eqnarray*}
\E\left[\left( \int_{-T}^T \mid Y(t) \mid^2 dt  \right)^{\frac{r}{2}}\right]  
& \leq & C.
\end{eqnarray*}

\textit{Step 2: BSVIEs with general delayed generator case}\\
Consider $ Y^0(s) = Z^0(t,s) = 0 $ and define recursively the sequence of processes $(Y^n,Z^n)_{n\geq 1}$ by: for a given $(Y^n,Z^n)\in  \mathbb{S}^2(\mathbb{R})\times\mathbb{H}^2(\mathbb{R})$,
\begin{eqnarray}\label{Eq8}
Y^{n+1}(t) & = & \psi(t) + \int_t^T f(t,s,Y^{n}_s,Z^{n}_{t,s} ) ds - \int_t^T Z^{n+1}(t,s) dW(s).
\end{eqnarray}
We need to establish for all $n\in \mathbb{N}^{\star}$ the following assertion  \\
$ (\mathcal{P}_n)
\left\lbrace  
\begin{array}{lcl}
\displaystyle \E \left[  \left( \int_0^T \mid Z^n(t,s)-Z^n(t',s) \mid^2 ds \right)^{\frac{r}{2}}\right]  & \leq & C_n \mid t-t' \mid^{\gamma r} ,   \\
\displaystyle \sup\limits_{t\in [-T,T]} \E \left[  \left( \int_0^T \mid Z^n(t,s) \mid^2 ds \right)^{\frac{r}{2}}\right]  & \leq & C_n, \\
\displaystyle \E\left[  \left( \int_{-T}^T \mid Y^n(t) \mid^2 dt \right)^{\frac{r}{2}}\right]  & \leq & C_n,   \\
t \mapsto Y^n(t) \;\; \mathrm{ is\; continuous} .  
\end{array}
\right. 
$\\
In virtue of the Step 1, $(\mathcal{P}_1) $ holds. Let suppose that there exists $n\in \mathbb{N}^{\star} $ such that $ (\mathcal{P}_n) $ is true. Let us prove the validity of $ (\mathcal{P}_{n+1})$. For this and like in Step 1 of existence's proof, let define
\begin{eqnarray*}
X_t^{n+1} (s) = \E \left(\psi(t) + \int_t^T f_n(t,u)du \mid \mathcal{F}_s  \right), 
\end{eqnarray*}
where $f_n(t,s) = f(t,s,Y_s^n,Z_{t,s}^n)$. It is easy to see that $ (X_t^{n+1}(s))_{s \geq 0} $ is continuous square integrable martingale. Hence the martingale representation provides that there exists a unique process $ Z^{n+1}(t,.) \in \mathbb{H}_2 $ such that
\begin{eqnarray*}
X_t^{n+1} (s) = X_t(0) + \int_0^s Z^{n+1}(t,u) dW(u),\;\; s \in [0,T].
\end{eqnarray*}
Moreover by Doob's maximal and Hölder's inequalities, we have for any $0\leq t<t'\leq T$,
\begin{eqnarray}\label{Eq9}
&&\E \left(  \sup\limits_{0\leq s\leq T} \mid X_t^{n+1} (s) - X_{t'}^{n+1} (s) \mid^r\right) \nonumber \\
& \leq &   C \E \left(  \mid X_t^{n+1} (T) - X_{t'}^{n+1} (T) \mid^r \right) \nonumber\\
& = &  C\E \left[ \left| \int_t^T f_n(t,s) ds - \int_{t'}^T f_n(t',s) ds   \right|^r  \right] \nonumber\\
& \leq &  C \E \left[ \left| \int_t^{t'} f_n(t,s) ds  \right|^r +  \left| \int_{t'}^T (f_n(t,s)  -  f_n(t',s)) ds   \right|^r   \right] \nonumber\\
& \leq & C  \mid t'-t \mid^{\frac{r}{2}} \E \left[ \left(\int_0^T |f_n(t,s)|^2ds \right)^{\frac{r}{2}  } \right]+C\E\left[ \left( \int_{t'}^T |f_n(t,s)  -  f_n(t',s)|^2 ds\right)^{r/2}\right].  
\end{eqnarray}
Next, it follows from assumptions $(\bf A2) $ and $(\bf A4)$ that 
\begin{eqnarray}\label{Z1}
\E\left(\int_0^T |f_n(t,s)|^2ds\right)&\leq & \E\left[\int^T_0\left(\int_{-T}^{0}(|Y^n(s+u)|^2+|Z^n(t+u,s+u)|^2)\alpha(du)\right)ds\right]\nonumber\\
&&+\E\left[\int_0^T |f(t,s,0,0)|^2ds\right]\nonumber\\
&\leq &\E\left(\int_{-T}^{T}|Y^n(s)|^2ds\right)+\sup_{0\leq t\leq T}\E\left(\int_{0}^{T}|Z^n(t,s)|^2)ds\right)\nonumber\\
&&+\E\left[\int_0^T |f(t,s,0,0)|^2ds\right]\nonumber\\
&\leq & C.
\end{eqnarray}
and
\begin{eqnarray}\label{Z2}
\E\left[ \left( \int_{t'}^T |f_n(t,s)  -  f_n(t',s)|^2 ds\right)^{r/2}\right]\leq C|t'-t|^{\gamma r}.
\end{eqnarray}
Finally, in view of recursive hypothesis, it follows from \eqref{Eq9} to \eqref{Z2} that 
\begin{eqnarray*}
\E \left(  \sup\limits_{0\leq s\leq T} \mid X_t^{n+1} (s) - X_{t'}^{n+1} (s) \mid^r\right)\leq C_n|t'-t|^{\gamma r}.
\end{eqnarray*}

Thanks to $ \gamma r > 1 $, Kolmogorov's criterion provides that $ (t,s) \mapsto X_t^{n+1} (s) $ is bi-continuous. In particular, $ t \mapsto X_t^{n+1}(t) = Y^{n+1}(t) $ is continuous.

On the other hand, by BDG's inequality together with Doob's maximal inequality and \eqref{Eq9} we have
\begin{eqnarray*}
\E \left[ \left( \int_0^T \mid Z^{n+1}(t,s)-Z^{n+1}(t',s) \mid^2 ds\right)^{\frac{r}{2}}\right]  & \leq & C \E\left[ \sup\limits_{0\leq u\leq T} \left| \int_0^u (Z^{n+1}(t,s)-Z^{n+1}(t',s)) dW(s)  \right|^r \right] \\
& \leq & C \E \left[ \sup\limits_{0\leq u\leq T} \left| X_t^{n+1} (u) - X_{t'}^{n+1} (u)\right|^r\right.\\
&&\left.  + \left| X_t^{n+1} (0) - X_{t'}^{n+1} (0)\right|^r   \right] \\
& \leq & C \mid t-t' \mid^{\gamma r}.
\end{eqnarray*}
Moreover, BDG's and Doob's maximal inequalities, assumptions $ ({\bf A2}) $ and $ ({\bf A4} ) $ together with $ (\mathcal{P}_n) $ derive
\begin{eqnarray*}
\E \left[ \left( \int_0^T \mid Z^{n+1}(0,s) \mid^2 ds\right)^{\frac{r}{2}}\right] 
& \leq & C \E\left[ \sup\limits_{0\leq u\leq T}  \left| \int_0^u Z^{n+1}(0,s) dW(s)  \right|^r \right]  \\
& \leq &  C \E\left[ \sup\limits_{0\leq u\leq T}  \left| X_0^{n+1} (u) - \E ( X_{0}^{n+1} (0))  \right|^r \right]  \\
& \leq &  C \E\left[ \left| X_0^{n+1} (T) -  X_{0}^{n+1} (0)  \right|^r \right]\\
& \leq &  C \E\left[ \mid \xi\mid^r + \left( \int_0^T \mid f(0,s,Y_s^n,Z_{0,s}^n \mid^2 ds  \right)^{\frac{r}{2}} \right]\\
& \leq & C \left( \E ( \mid \xi\mid^r ) + \E \left[ \left( \int_{-T}^T \mid Y^n(s) \mid^2 ds   \right)^{\frac{r}{2}} \right] \right. \\
&& \left.  + \sup\limits_{0\leq t\leq T} \E \left[ \left( \int_{-T}^T \mid Z^n(0,s) \mid^2 ds   \right)^{\frac{r}{2}} \right] \right. \\
&& \left. + \E \left[ \left( \int_{0}^T \mid f(0,s,0,0) \mid^2 ds   \right)^{\frac{r}{2}} \right]  \right)\\
&& \leq  C_{n+1}.   
\end{eqnarray*}
Next, since $ Y^{n+1} (t) = X_t^{n+1}(t) $, following the same computation as above we get
\begin{eqnarray*}
\E \left[ \left( \int_{-T}^T \mid Y^{n+1}(t) \mid^2 dt\right)^{\frac{r}{2}}\right]
& \leq & C_{n+1},
\end{eqnarray*}
which proves $( \mathcal{P}_{n+1}) $. Finally, it follows from recurrence principle that for all $ n \in \mathbb{N}^{\star} $, the function $t\rightarrow Y_t ^n $ is continuous a.s.

We have yet to establish that the sequence $ Y^n$ converges in the space of continuous functions. To make this, similar to Step 3 of the proof of Theorem \ref{T}, we have
\begin{eqnarray*}
\E \left( \sup\limits_{0\leq t\leq T} \mid Y^{n+1} (t) - Y^n (t) \mid^r   \right) 
& \leq C (\beta )^{rn},
\end{eqnarray*} 
where $C$ is independent of $n$ and $0<\beta<1$. Hence, there exists a continuous adapted process $\bar{Y}$ such that
\begin{eqnarray*}
\lim\limits_{n \rightarrow + \infty} \E \left( \sup\limits_{0\leq t\leq T} \mid Y^{n+1} (t) -\bar{Y} (t) \mid^r   \right) = 0 .
\end{eqnarray*}
Next, taking the limit in equation \eqref{Eq8}, it follows that $ \bar{Y}$ solves BSVIE \eqref{Eq1}. Finally by the uniqueness of solution, we obtain $Y=\bar{Y}$ and then the map $t \rightarrow Y(t)$ is continuous.
\end{proof}
\begin{remark}
As focused in introduction, this paper is the first in a series of several futur works. Indeed, this current work allowed us to derive maximum principle for optimal control related to FVSVIE with delayed generator. Then, like in \cite{P1,P'1}, we will apply delayed BSVIEs for pricing, hedging and portfolio management in insurance and finance.	
\end{remark}

\label{lastpage-01}
\end{document}